\documentclass[reqno]{amsart}
\usepackage[latin1]{inputenc}

\usepackage{enumerate}
\usepackage{amstext,amsmath,amsthm,amsfonts,amssymb,amscd}
\usepackage{latexsym,mathrsfs,eufrak,dsfont,euscript}
\usepackage{mathtools}
\usepackage{stmaryrd}

\usepackage{subfig}
\usepackage{hyperref}
\hypersetup{
    colorlinks=true%
    citecolor=blue,%
    filecolor=black,%
    linkcolor=red,%
    urlcolor=blue
}

\usepackage{color}
\usepackage{tikz}
\usepackage{pgfplots}
\usepackage{graphicx}

\usetikzlibrary{plotmarks}
\usetikzlibrary{calc,through,backgrounds}
\usetikzlibrary{calendar,snakes}
\usetikzlibrary{positioning}

\usepackage[normalem]{ulem}
\usepackage{cancel} 

\newtheorem{theorem}{Theorem}

\newtheorem{lemma}[theorem]{Lemma}
\newtheorem{corollary}[theorem]{Corollary}
\theoremstyle{definition}

\theoremstyle{remark}

\numberwithin{equation}{section}
\DeclareMathOperator{\supp}{supp}
\newcommand{\abs}[1]{\left\vert#1\right\vert}
\newcommand{\set}[1]{\left\{#1\right\}}
\newcommand{\proin}[2]{\left<#1,#2\right>}
\newcommand{\norm}[1]{\left\Vert#1\right\Vert}
\newcommand{\R}{\mathbb{R}}
\newcommand{\Dy}{\mathfrak{D}}
\begin{document}
\title[]{On dyadic nonlocal Schr\"{o}dinger equations with Besov initial data}
\author[]{Hugo Aimar}

\author[]{Bruno Bongioanni}

\author[]{Ivana G\'{o}mez}
%
%

\thanks{The research was supported  by CONICET, ANPCyT (MINCyT), and UNL}

\begin{abstract}
In this paper we consider the pointwise convergence to the initial data for the Schr\"{o}dinger-Dirac equation
$i\tfrac{\partial u}{\partial t}=D^{\beta}u$ with $u(x,0)=u^0$ in a dyadic Besov space. Here $D^{\beta}$ denotes the fractional derivative of order $\beta$ associated to the dyadic distance $\delta$ on $\mathbb{R}^+$. The main tools are a sumability formula for the kernel of $D^{\beta}$ and pointwise estimates of the corresponding maximal operator in terms of the dyadic Hardy-Littlewood function and the Calder\'on sharp maximal operator.
\end{abstract}
\maketitle
\section{Introduction}

In quantum mechanics time dependent Schr\"{o}dinger type equations with space derivatives of order less than two, have been considered since the introduction of the Dirac operator which is actually local and of first order \cite{Dirac1928}. More recently some fractional nonlocal Riemann-Liouville calculus, and some other nonlocal cases, have also been considered in the literature, \cite{Laskin02}. See also \cite{Muslih2010}, \cite{MaSkuKro2011} and \cite{ChenGuo2007}.

\smallskip
The differential operator in the space variable that we shall consider is an analogous of the nonlocal fractional derivative of order $\beta>0$
\begin{equation}\label{eq:fractionalderivative}
\int\frac{f(x)-f(y)}{\abs{x-y}^{1+\beta}}dy.
\end{equation}

\smallskip
The basic difference is given by the fact that we substitute the Euclidean distance $\abs{x-y}$ by the dyadic distance $\delta$ from $x$ to $y$. To introduce our main result let us start by defining the basic metric $\delta$ and the Besov type spaces induced by $\delta$ on the interval $\mathbb{R}^+=(0,\infty)$.

\smallskip
Let $\Dy=\cup_{j\in \mathbb{Z}}\Dy^j$ be the family of the standard dyadic intervals in $\mathbb{R}^+$. In other words $I\in\Dy$ if $I=I^j_k=[(k-1)2^{-j},k2^{-j})$, $j\in \mathbb{Z}$, $k\in \mathbb{Z}^+$. Each $\Dy^j$ contains the intervals of the $j$-th level, for $I\in\Dy^j$, $\abs{I}=2^{-j}$. We shall write $\Dy^+$ to denote the intervals $I$ in $\Dy$ with $\abs{I}\leq 1$. For $I\in\Dy^j$ we shall denote by $I^+$ and $I^-$ the right and left halves of $I$, which belong to $\Dy^{j+1}$.  Given two points $x$ and $y$ in $\mathbb{R}^+$ its dyadic distance $\delta(x,y)$, is defined as the length of the smallest dyadic interval $J\in\Dy$ which contains $x$ and $y$. On the diagonal $\Delta$ of $\mathbb{R}^+\times \mathbb{R}^+$, $\delta$ vanishes.

\smallskip
Since for $x$ fixed $\delta^{-1-\beta}(x,y)$ is not integrable, the analog to \eqref{eq:fractionalderivative} with $\delta(x,y)$ instead of $\abs{x-y}$ in $\mathbb{R}^+$ is well defined as an absolutely convergent integral, only on a subspace of functions which have certain regularity with respect to the distance $\delta$.
For $0<\lambda<1$, with $B^{\lambda}_{2, dy}$ we denote the class of all $L^2$ complex valued  functions $f$ defined on $\mathbb{R}^+$ such that
\begin{equation*}
\iint_{Q}\frac{\abs{f(x)-f(y)}^2}{\delta(x,y)^{1+2\lambda }} dx dy<\infty,
\end{equation*}
with $Q=\set{(x,y)\in \mathbb{R}^2: \delta(x,y)<2}$. For $f$ and $g$ both in $B^{\lambda}_{2,dy}$, the inner product
\begin{equation*}
\int_{\mathbb{R}^+}f\overline{g} dx+ \iint_{Q} \frac{f(x)-f(y)}{\delta(x,y)^{\lambda}}\frac{\overline{g(x)-g(y)}}{\delta(x,y)^{\lambda}}\frac{dx dy}{\delta(x,y)},
\end{equation*}
gives a Hilbert structure on $B^{\lambda}_{2,dy}$.

\smallskip
Since, as it is easy to check from the definition of $\delta$, $\abs{x-y}\leq \delta (x,y)$ when $(x,y)\in Q$, we have that the standard Besov space $B^{\lambda}_{2}$ on $\mathbb{R}^+$ is a subspace of $B^{\lambda}_{2, dy}$. See \cite{Peetre76} for the classical theory of Besov spaces.

\smallskip
For $I\in\Dy$ we shall write $h_I$ to denote the Haar wavelet adaptated to $I$. In other words $h_I=\abs{I}^{-\tfrac{1}{2}}(\mathcal{X}_{I^{-}}-\mathcal{X}_{I^{+}})$ where, as usual $\mathcal{X}_E$ is the indicator function of the set $E$. Sometimes, when the parameters of scale and position $j$ and $k$, need to be emphasized, we shall write $h^j_k$ to denote $h_I$ for $I=I^j_k$. In the sequel the scale parameter $j$ of $I$ will be denoted  by $j(I)$. As it is well known $\{h_I: I\in\Dy\}$ is an orthonormal basis for $L^2$. As usual we write $V_0$ to denote the subspace of $L^2$ of those functions which are constant on each interval between integers. With $P_0$ we denote the projector of $L^2$ onto $V_0$.

\smallskip
As a consequence of Theorem \ref{thm:mainresult} in Section 3, we shall obtain the next result.
\begin{theorem}\label{thm:consequencemainresult}
Let $0<\beta<1$  and $u^0\in L^2$ with $P_0 u^0=0$, be given. Assume that $u^0$ is a function in $B^{\lambda}_2$ with $\beta<\lambda<1$, then  the function defined by
\begin{equation*}
u(x,t)=\sum_{I\in\Dy}e^{i t \abs{I}^{-\beta}}\proin{u^0}{h_I}h_I(x)
\end{equation*}
solves the problem
\begin{equation*}
(P)\quad  \begin{dcases}
i \frac{\partial u}{\partial t}(x,t) = \frac{2^\beta-1}{2^\beta}\int_{\mathbb{R}^+}\frac{u(x,t)-u(y,t)}{\delta(x,y)^{1+\beta}}\,dy  & \mbox{$x\in \mathbb{R}^+, t>0$ }\\
\medskip
 u(x,0)=u^0(x)  & \mbox{$x\in \mathbb{R}^+$;}
\end{dcases}
\end{equation*}
where the initial condition is verified pointwise almost everywhere.
\end{theorem}

\smallskip
The identification of function spaces with low regularity, for which the pointwise convergence to the initial data for solutions of the time dependent free particle Schr\"{o}dinger equation is a hard problem. Some basic fundamental steps in this direction are contained in  \cite{Carleson80}, \cite{DaKe82}, \cite{KeRui83}, \cite{Cow83}, \cite{Sjo87}, \cite{Vega88}, \cite{TaVa2000}.

\smallskip
The paper is organized as follows. In Section~\ref{sec:dyadicBesovspaces} we introduce the basic operator and the corresponding Besov space and its wavelet characterization in terms of the Haar system. In Section~\ref{sec:mainresultproof} we prove the main result, which contains a detailed formulation of Theorem~\ref{thm:consequencemainresult}.

\section{Nonlocal dyadic differential operators and dyadic Besov spaces}\label{sec:dyadicBesovspaces}

Let $0<\beta<1$ be given. We shall deal with the operator $D^{\beta}$ whose spectral form in the Haar system is given by $D^{\beta} h_I=\abs{I}^{-\beta}h_I$ for $I\in\Dy$.

Let $\mathcal{S}(\mathscr{H})$ be the linear span of the Haar system $\mathscr{H}=\{h_I: I\in\Dy\}$. The space $\mathcal{S}(\mathscr{H})$ is dense in $L^2$.

The operator $D^{\beta}$ is well defined from $\mathcal{S}(\mathscr{H})$ into itself and is given by
\begin{equation*}
D^{\beta}f=\sum_{I\in\Dy}\abs{I}^{-\beta}\proin{f}{h_I}h_I
\end{equation*}
for $f\in\mathcal{S}(\mathscr{H})$. Observe that $D^{\beta}$ is unbounded in the $L^2$ norm.

In the next result we show that $D^{\beta}$ has the structure of a nonlocal differential operator if we change the Euclidean distance by the dyadic distance on $\mathbb{R}^+$.

\begin{theorem}\label{thm:integralformofoperatorinlinearspan}
Let $0<\beta<1$ be given, then for $f\in\mathcal{S}(\mathscr{H})$ we have
\begin{equation}\label{eq:operatorDbetadiffoperator}
D^{\beta}f(x)=\frac{2^\beta-1}{2^\beta}\int_{\mathbb{R}^+}\frac{f(x)-f(y)}{\delta(x,y)^{1+\beta}}dy,
\end{equation}
where the integral on the right hand side is absolutely convergent.
\end{theorem}

Before proving Theorem~\ref{thm:integralformofoperatorinlinearspan}, we collect some basic properties of $\delta$.

\begin{lemma}\label{lem:geometricpropertiesdelta}\quad
\begin{enumerate}[(\ref{lem:geometricpropertiesdelta}.a)]
\item $\mathbb{R}^+\times\mathbb{R}^+$ is the disjoint union of the diagonal $\Delta$ and the level sets $\Lambda_j=\{(x,y)\in \mathbb{R}^+\times\mathbb{R}^+: \delta(x,y)=2^{-j}\}$ of $\delta$ for $j\in\mathbb{Z}$, and $Q=\cup_{j\geq 0}\Lambda_j$. (See Figure \ref{fig:levelsetsdelta}).\label{item:uno}

\item For $\gamma\in\R$, $\delta^{\gamma}=\sum_{j\in\mathbb{Z}}2^{-j\gamma}\mathcal{X}_{\Lambda_j}$.\label{item:dos}

\item Each $\Lambda_j$ is the disjoint union of the sets $B(I)=(I^+\times I^-)\cup(I^-\times I^+)$ for $I\in\Dy^j$.\label{item:tres}

\item For $f\in\mathcal{S}(\mathscr{H})$, set $F(x,y)=f(x)-f(y)$, then $\inf\{\delta(x,y):(x,y)\in\supp F\}>0$. \label{item:cuatro}

\item Let $\alpha>-1$. Then for every $x\in\mathbb{R}^+$, $\delta(x,y)^{\alpha}$ is locally integrable as a function of $y$. Moreover,  $\int_{l-1}^l\delta(x,y)^{\alpha}dy$ is bounded by $(2^{1+\alpha}-1)^{-1}$ for every $l\in\mathbb{Z}^+$. \label{item:cinco}

\end{enumerate}
\end{lemma}


\begin{figure}[htp]
\begin{center}
\begin{tikzpicture}[scale=2]
\foreach \x in {0}
\draw[-] (0+\x,0+\x)--(1+\x,0+\x)--(1+\x,1+\x)--(0+\x,1+\x)--(0+\x,0+\x);
\draw[-] (1,1)--(1,1.2);
\draw[-] (1,1)--(1.2,1);
\draw[-] (0,0)--(-.2,0);
\draw[-] (0,0)--(0,-.2);
\draw[line width=.5pt, dotted] (-.2,-0.2)--(1.2,1.2);
\draw (1.3,1.3)node[] {\scalebox{1}{$\displaystyle \Delta$}};
\foreach \x in {0}
\draw[fill=lightgray!80] (.5+\x,0+\x)--(1+\x,0+\x)--(1+\x,.5+\x)--(.5+\x,.5+\x)--(.5+\x,0+\x);
\foreach \x in {0}
\draw[fill=lightgray!80] (0+\x,0.5+\x)--(.5+\x,.5+\x)--(.5+\x,1+\x)--(0+\x,1+\x)--(0+\x,.5+\x);
\foreach \x in {0}
\draw[fill=black!60] (.25+\x,0+\x)--(.5+\x,0+\x)--(.5+\x,.25+\x)--(.25+\x,.25+\x)--(.25+\x,0+\x);
\foreach \x in {0}
\draw[fill=black!60] (0+\x,.25+\x)--(.25+\x,.25+\x)--(.25+\x,.5+\x)--(0+\x,.5+\x)--(0+\x,.25+\x);
\foreach \x in {0}
\draw[fill=black!60] (0.75+\x,.5+\x)--(1+\x,.5+\x)--(1+\x,.75+\x)--(.75+\x,.75+\x)--(0.75+\x,.5+\x);
\foreach \x in {0}
\draw[fill=black!60] (0.5+\x,.75+\x)--(.75+\x,.75+\x)--(.75+\x,1+\x)--(.5+\x,1+\x)--(0.5+\x,.75+\x);
\foreach \x in {0,1/2}
\draw[fill=black!90] (.25/2+\x,0+\x)--(.5/2+\x,0+\x)--(.5/2+\x,.25/2+\x)--(.25/2+\x,.25/2+\x)--(.25/2+\x,0+\x);
\foreach \x in {0,1/2}
\draw[fill=black!90] (0+\x,.25/2+\x)--(.25/2+\x,.25/2+\x)--(.25/2+\x,.5/2+\x)--(0+\x,.5/2+\x)--(0+\x,.25/2+\x);
\foreach \x in {0,1/2}
\draw[fill=black!90] (0.75/2+\x,.5/2+\x)--(1/2+\x,.5/2+\x)--(1/2+\x,.75/2+\x)--(.75/2+\x,.75/2+\x)--(0.75/2+\x,.5/2+\x);
\foreach \x in {0,1/2}
\draw[fill=black!90] (0.5/2+\x,.75/2+\x)--(.75/2+\x,.75/2+\x)--(.75/2+\x,1/2+\x)--(.5/2+\x,1/2+\x)--(0.5/2+\x,.75/2+\x);

\end{tikzpicture}
\end{center}
\caption{The picture depicts schematically the level sets $\Lambda_j$ of $\delta$ for $j=0$ (lightgray), for $j=1$ (darkgray) and $j=2$ (black).}\label{fig:levelsetsdelta}
\end{figure}
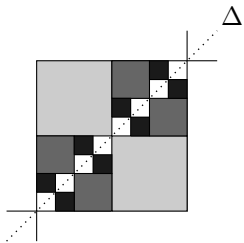

\begin{proof}[Proof of Lemma~\ref{lem:geometricpropertiesdelta}]\quad

{\it Proof of (\ref{lem:geometricpropertiesdelta}.\ref{item:uno}).\,}
 Given a point $(x,y)\in \mathbb{R}^+\times\mathbb{R}^+$ which does not belong to $\Delta$, since for some $J\in\Dy$, $(x,y)\in J\times J$ and since $x\neq y$, there exists one and only one subinterval $I$ of $J$ such that $x$ and $y$ belong both to $I$ but not both to the same half of $I$. In other words $(x,y)\in B(I)$. Since $I\subset J$ then, $j(I)\geq 0$ and $\delta (x,y)=2^{-j(I)}$, so that $(x,y)\in \Lambda_{j(I)}$.

\smallskip
{\it Proof of (\ref{lem:geometricpropertiesdelta}.\ref{item:dos}).\,}
Follows directly from \textit{(\ref{lem:geometricpropertiesdelta}.\ref{item:uno})}.

\smallskip
{\it Proof of (\ref{lem:geometricpropertiesdelta}.\ref{item:tres}).\,}
Notice first that if $I$ and $J$ are two different intervals on $\Dy^j$, then $I^+\cap J^+=\emptyset$ and $I^-\cap J^-=\emptyset$ hence $B(I)\cap B(J)=\emptyset$. On the other hand, if $(x,y)\in B(I)$ for some $I\in\Dy^j$, then $x\in I^+$ and $y\in I^-$ or $x\in I^-$ and $y\in I^+$, so that the smallest dyadic interval containing both $x$ and $y$ is $I$ itself. This means that $\delta (x,y)=2^{-j}$, in other words $(x,y)\in \Lambda_j$. Assume now that $(x,y)$ is any point in $\Lambda_j$, then $\delta (x,y)=2^{-j}$. This means that there exists $I\in\Dy^j$ such that $(x,y)\in I\times I$ but $x$ and $y$ do not belong to the same half of $I$. In other words $(x,y)\in I\times I$ but $(x,y)\notin (I^-\times I^-)\cup (I^+\times I^+)$. Hence $(x,y)\in B(I)$.

\smallskip
{\it Proof of (\ref{lem:geometricpropertiesdelta}.\ref{item:cuatro}).\,}
Since any $f\in\mathcal{S}(\mathscr{H})$ is finite linear combination of some of the $h_I$'s, all we need to prove is that $\inf\{\delta(x,y): (x,y)\in\supp H_I\}>0$ for every $I\in\Dy$, where $H_I(x,y)=h_I(x)-h_I(y)$. Take $I\in\Dy$, then $I\in\Dy^j$ for some $j\geq 0$ and $H_I$ vanishes on $(I^-\times I^-)\cup (I^+\times I^+)$ hence $\delta(x,y)\geq 2^{-j}$ for every $(x,y)\in\supp H_I$.

\smallskip
{\it Proof of (\ref{lem:geometricpropertiesdelta}.\ref{item:cinco}).\,}
The desired properties are trivial for $\alpha\geq 0$. Assume then that  $-1<\alpha<0$ and $x\in\mathbb{R}^+$. Then $\int_{l-1}^l\delta(x,y)^\alpha dy$ vanishes when $x\notin (l-1,l)$. If $x\in (l-1,l)$ then
\begin{align*}
\int_{l-1}^l\delta(x,y)^{\alpha}dy &= \sum_{k=0}^{\infty}\int_{\{y\in (l-1,l): 2^{-k-1}<\delta(x,y)<2^{-k}\}}\delta(x,y)^{\alpha}dy\\
&\leq \sum_{k=0}^{\infty}2^{-\alpha(k+1)}\abs{\{y\in(l-1,l):\delta(x,y)<2^{k}\}}\\
&\leq \sum_{k=0}^{\infty}2^{-(1+\alpha)(k+1)}=(2^{1+\alpha}-1)^{-1}.
\end{align*}
\end{proof}

\begin{proof}[Proof of Theorem~\ref{thm:integralformofoperatorinlinearspan}]
It is enough  to check \eqref{eq:operatorDbetadiffoperator} for $f=h_I$. From \textit{(\ref{lem:geometricpropertiesdelta}.\ref{item:dos})} and \textit{(\ref{lem:geometricpropertiesdelta}.\ref{item:tres})} we have

\begin{align*}
 \int_{\mathbb{R}^+}\frac{h_I(x)-h_I(y)}{\delta(x,y)^{1+\beta}} dy
&=\int_{\mathbb{R}^+}\Bigl(\sum_{j\in\mathbb{Z}}2^{j(1+\beta)}\mathcal{X}_{\Lambda_j}(x,y)\Bigr)(h_I(x)-h_I(y)) dy\\
& =\sum_{j\in\mathbb{Z}}2^{j(1+\beta)}\int_{\mathbb{R}^+}\mathcal{X}_{\Lambda_j}(x,y)(h_I(x)-h_I(y)) dy\\
&=\sum_{j\in\mathbb{Z}}2^{j(1+\beta)}\sum_{J\in\Dy^j}\int_{\mathbb{R}^+}\mathcal{X}_{B(J)}(x,y)(h_I(x)-h_I(y)) dy.
\end{align*}
Now, since the support of $h_I(x)-h_I(y)$ and $B(J)$ are disjoint when $j(I)<j$ the last sum of $j$ reduces to the sum for $j\leq j(I)$. On the other hand, for $j\leq j(I)$ there exists a unique $J_j\in\Dy^j$ such that the support of $(h_I(x)-h_I(y))$ intersects $B(J)$. Actually that unique  $J_j$ is the only ancestor of $I$ in the generation $j$. With these remarks in mind we have, for $x\in I^{-}$, that
\begin{align*}
 \int_{\mathbb{R}^+}\frac{h_I(x)-h_I(y)}{\delta(x,y)^{1+\beta}} dy
&=\abs{I}^{-\tfrac{1}{2}}\sum_{j\leq j(I)}2^{j(1+\beta)}\int_{J_j}[1-\mathcal{X}_{I^{-}}(y)+\mathcal{X}_{I^{+}}(y)] dy\\
& = \abs{I}^{-\tfrac{1}{2}}\sum_{j\leq j(I)}2^{j(1+\beta)}\abs{J_j}=\frac{2^\beta}{2^\beta-1}\abs{I}^{-\beta}\abs{I}^{-\tfrac{1}{2}}.
\end{align*}
In a similar way, with $x\in I^{+}$, we get that $\int_{\mathbb{R}^+}\frac{h_I(x)-h_I(y)}{\delta(x,y)^{1+\beta}} dy=-\frac{2^\beta}{2^\beta-1}\abs{I}^{-\beta}\abs{I}^{-\tfrac{1}{2}}$. In other words, $\int_{\mathbb{R}^+}\frac{h_I(x)-h_I(y)}{\delta(x,y)^{1+\beta}} dy=\frac{2^\beta}{2^\beta-1}\abs{I}^{-\beta}h_I(x)$, as desired.
\end{proof}

A basic identity to obtain a characterization of the Besov type spaces in terms of the Haar system is contained in Theorem~\ref{thm:besovdyaigualhaar}.

\begin{theorem}\label{thm:besovdyaigualhaar}
Let  $0<\lambda<1$, be given, then the identity
\begin{equation}\label{eq:besovdyaigualhaar}
\iint_{Q}\frac{\abs{f(x)-f(y)}^2}{\delta(x,y)^{1+2\lambda}}dx dy=\sum_{I\in\Dy^+}\abs{\proin{f}{h_I}}^2[(2+c_{\lambda})\abs{I}^{-2\lambda}-c_{\lambda}]
\end{equation}
holds for every function $f\in\mathcal{S}^+(\mathscr{H})$ and $c_{\lambda}=2(2^{2\lambda}-1)^{-1}$, where $\mathcal{S}^+(\mathscr{H})$ is the linear span of $\set{h_I: I\in\Dy^+}$.
\end{theorem}

Theorem~\ref{thm:besovdyaigualhaar} will be a consequence of some elementary geometric properties of the dyadic system and the distance $\delta$.

\begin{lemma}\label{lem:geometricproperties}\quad
\begin{enumerate}[(\ref{lem:geometricproperties}.a)]
\item Set $C(J)=[(J\times \mathbb{R}^+)\cup (\mathbb{R}^+\times J)]\setminus(J\times J)$ for $J\in\Dy^+$, then $B(I)\cap C(J)=\emptyset$ for $j(I)\geq j(J)$. \label{item:seis}

\item For every $I\in\Dy^+$ and every $j=0,1,\ldots,j(I)-1$ there exists one and only one $J\in\Dy^j$ for which $B(J)$ intersects $C(I)$. \label{item:siete}

\item For each $I\in\Dy^+$ we have
\begin{equation*}
\sum_{j\geq 0}2^{j(1+2\lambda)}\sum_{J\in\Dy^j}m(B(J)\cap C(I))=c_{\lambda}\abs{I}( \abs{I}^{-2\lambda}-1),
\end{equation*}
where $m$ is the area measure in $\R^2$, $\lambda>0$ and $c_{\lambda}=2(2^{2\lambda}-1)^{-1}$.\label{item:ocho}

\item For $j\geq 0$, $I\in\Dy^+$ and $J\in\Dy^+$, with $I\neq J$,
\begin{equation*}
\mathcal{I}(j,I,J):= \iint_Q \mathcal{X}_{\Lambda_j}(x,y)[h_I(x)-h_I(y)][h_J(x)-h_J(y)] dx dy=0.
\end{equation*}\label{item:nueve}

\item For each $I\in\Dy^+$,
\begin{equation*}
\sum_{j\geq 0}2^{j(1+2\lambda)}\mathcal{I}(j,I,I)= (2+c_{\lambda})\abs{I}^{-2\lambda}-c_{\lambda}.
\end{equation*} \label{item:diez}
\end{enumerate}
\end{lemma}

Let us start by proving Theorem~\ref{thm:besovdyaigualhaar} assuming the results in Lemma~\ref{lem:geometricproperties}.
\begin{proof}[Proof of Theorem \ref{thm:besovdyaigualhaar}]
Let $f$ be a finite linear combination of some of the Haar functions $h_I$ for $I\in\Dy^+$, i.e. $f=\sum_{I\in\Dy^+}\proin{f}{h_I}h_I$ with $\proin{f}{h_I}=0$ except for a finite number of $I$ in $\Dy^+$. From \textit{(\ref{lem:geometricpropertiesdelta}.\ref{item:uno})} and \textit{(\ref{lem:geometricpropertiesdelta}.\ref{item:dos})} in Lemma~\ref{lem:geometricpropertiesdelta}, \textit{(\ref{lem:geometricproperties}.\ref{item:ocho})}, \textit{(\ref{lem:geometricproperties}.\ref{item:nueve})} and \textit{(\ref{lem:geometricproperties}.\ref{item:diez})} in Lemma~ \ref{lem:geometricproperties} we get

\begin{align*}
\iint_{Q} \frac{\abs{f(x)-f(y)}^2}{\delta(x,y)^{1+2\lambda}} & dx dy\\
&=\iint_{Q}\Bigl(\sum_{j\geq 0}2^{j(1+2\lambda)}\mathcal{X}_{\Lambda_j}(x,y)\Bigr)\Bigl(\sum_{I\in\Dy^+}\sum_{J\in\Dy^+}\proin{f}{h_I}\proin{f}{h_J} \\
& \phantom{=\iint_{Q}\Bigl(\sum_{j\geq 0}2^{j(1+2\lambda)}\mathcal{X}}
\times[h_I(x)-h_I(y)][h_J(x)-h_J(y)]\Bigr)dx dy\\
&=\sum_{j\geq 0}2^{j(1+2\lambda)}\sum_{I\in\Dy^+}\sum_{J\in\Dy^+}\proin{f}{h_I}\proin{f}{h_J}\iint_{Q}\mathcal{X}_{\Lambda_j}(x,y) \\
& \phantom{=\sum_{j\geq 0}2^{j(1+2\lambda)}\sum_{I\in\Dy^+}\sum_{J\in\Dy^+}}
\times[h_I(x)-h_I(y)][h_J(x)-h_J(y)] dx dy\\
&= \sum_{j\geq 0}2^{j(1+2\lambda)}\sum_{I\in\Dy^+}\abs{\proin{f}{h_I}}^2\iint_{\Lambda_j}[h_I(x)-h_I(y)]^2 dx dy \\
&= \sum_{I\in\Dy^+}\abs{\proin{f}{h_I}}^2 \sum_{j\geq 0}2^{j(1+2\lambda)}\iint_{\Lambda_j}[h_I(x)-h_I(y)]^2 dx dy\\
&= \sum_{I\in\Dy^+}\abs{\proin{f}{h_I}}^2 \sum_{j\geq 0}2^{j(1+2\lambda)}\mathcal{I}(j,I,I)\\
&= \sum_{I\in\Dy^+}\abs{\proin{f}{h_I}}^2[(2+c_{\lambda})\abs{I}^{-2\lambda}-c_{\lambda}].
\end{align*}
\end{proof}

\begin{proof}[Proof of Lemma~\ref{lem:geometricproperties}]\quad

{\it Proof of (\ref{lem:geometricproperties}.\ref{item:seis}).\,}
Since  $j(I)\geq j(J)$ we have $I\subseteq J$ or $I\cap J=\emptyset$, we divide our analysis in these two cases. When $I\cap J=\emptyset$, then $I^+\cap J=\emptyset$ and $I^-\cap J=\emptyset$ and $B(I)\cap C(J)=\emptyset$. Assume now that $I\subseteq J$, then $B(I)\subset J\times J$ which is disjoint from $C(J)$.

\smallskip
{\it Proof of (\ref{lem:geometricproperties}.\ref{item:siete}).\,}
Let $I\in\Dy^+$ and $j=0,1,\ldots,j(I)-1$ be given. Let $J$ be the only dyadic interval in $\Dy^j$ such that $J\supsetneqq I$. Then $C(I)\cap B(J)\neq\emptyset$. In fact, since $J\supsetneqq I$, then $I\subset J^+$ or $I\subset J^-$. Assume for example that $I\subset J^+$, then any point $(x,y)$ with $x\in I$ and $y\in J^-$ belongs to both $C(I)$ and $B(J)$. So that,  since for $J\in\Dy^j$ and $j<j(I)$, arguing as in the proof of \textit{(\ref{lem:geometricproperties}.\ref{item:seis})}, the condition $J\supset I$ is necessary for $B(J)\cap C(I)\neq \emptyset$, we get the result.

\smallskip
{\it Proof of (\ref{lem:geometricproperties}.\ref{item:ocho}).\,}
Let $I\in\Dy^+$ be given. For $j=0,1,\ldots,j(I)-1$ set $J(j,I)$ to denote the only $J\in\Dy^j$ for which $B(J)\cap C(I)\neq \emptyset$, provided by \textit{(\ref{lem:geometricproperties}.\ref{item:siete})}. Now from \textit{(\ref{lem:geometricproperties}.\ref{item:seis})}
we have
\begin{align*}
\sum_{j\geq 0}2^{j(1+2\lambda)}\sum_{J\in\Dy^j} m(B(J)\cap C(I))
&= \sum_{j=0}^{j(I)-1}2^{j(1+2\lambda)}\sum_{J\in\Dy^j} m(B(J)\cap C(I))\\
&= \sum_{j=0}^{j(I)-1}2^{j(1+2\lambda)}m(B(J(j,I))\cap C(I)).
\end{align*}
But, as it is easy to see, $m(B(J(j,I))\cap C(I))=2\abs{I}2^{-j}$. Hence
\begin{equation*}
\sum_{j\geq 0}2^{j(1+2\lambda)}\sum_{J\in\Dy^j} m(B(J)\cap C(I))
= 2\abs{I}\sum_{j=0}^{j(I)-1}2^{j(1+2\lambda)}2^{-j}
= c_\lambda \abs{I} (\abs{I}^{-2\lambda}-1).
\end{equation*}

\smallskip
{\it Proof of (\ref{lem:geometricproperties}.\ref{item:nueve}).\,}
From \textit{(\ref{lem:geometricpropertiesdelta}.\ref{item:tres})} it is enough to show that $\iint_{B(K)}k_{IJ}(x,y)dxdy=0$ for every $I, J$ and $K\in\Dy^+$ with $I\neq J$, where $k_{I J}(x,y)=(h_I(x)-h_I(y))(h_J(x)-h_J(y))$. We shall divide our analysis into two cases according to the relative positions of $I$ and $J$.

Assume first that $I\cap J=\emptyset$, more precisely, assume that $I$ is to the left of $J$. Then
$k_{IJ}(x,y)\!=\!\!\sqrt{\abs{I}\abs{J}}\!
\left[(\mathcal{X}_{I^-\times J^+}(x,y)-\mathcal{X}_{I^-\times J^-}(x,y)+\mathcal{X}_{I^+\times J^-}(x,y)-\mathcal{X}_{I^+\times J^+}(x,y))\right.$
$+\left.\!(\mathcal{X}_{J^-\times I^+}(x,y)\!-\mathcal{X}_{J^-\times I^-}(x,y)+\mathcal{X}_{J^+\times I^-}(x,y)\!-\mathcal{X}_{J^+\times I^+}(x,y))
\right]$
whose su\-pport is $(I\times J)\cup (J\times I)$. See Figure \ref{fig:IvacioJ}.


\begin{figure}[htp]
\begin{center}
\begin{tikzpicture}[scale=2]
\draw[line width=.9pt, color= gray!80] (1,0)--(1.25,0); 
\draw[line width=.5pt, snake=brace] (1.25,-.08)--(1,-.08);
\node at (1.1,-0.25) {\scalebox{1}{$J$}};
\draw[line width=.9pt, color= black!80] (.25,0)--(.25+.25/2,0); 
\draw[line width=.5pt, snake=brace] (-.08,.25-1/64)--(-.08,.25+.25/2+1/64);
\node at (-.2,.3) {\scalebox{1}{$I$}};
\draw[line width=.9pt, color= gray!50] (1,0)--(1.25,0)--(1.25,2)--(1,2)--(1,0); 
\draw[line width=.9pt, color= gray!50] (0,1)--(0,1.25)--(2,1.25)--(2,1)--(0,1); 
\draw[line width=.9pt, color= gray!50] (.25,0)--(.25+.25/2,0)--(.25+.25/2,2)--(.25,2)--(.25,0); 
\draw[line width=.9pt, color= gray!50] (0,.25)--(0,.25+.25/2)--(2,.25+.25/2)--(2,.25)--(0,.25); 
%
\draw[line width=.9pt, color= black!80, fill= yellow] (.25,1+.125)--(.25,1.25)--(.25+.25/4,1.25)--(.25+.25/4,1+.125)--(.25,1+.125);
\draw[line width=.9pt, color= black!80, fill= yellow] (.25+.25/4,1+.125)--(.25+.25/2,1+.125)--(.25+.25/2,1)--(.25+.25/4,1)--(.25+.25/4,1+.125);
\draw[line width=.9pt, color= black!80, fill= orange] (.25,1)--(.25,1+.25/2)--(.25+.25/4,1+.125)--(.25+.25/4,1)--(.25,1);
\draw[line width=.9pt, color= black!80, fill= orange] (.25+.25/4,1+.125)--(.25+.25/4,1.25)--(.25+.25/2,1.25)--(.25+.25/2,1+.125)--(.25+.25/4,1+.125);
\draw[line width=.9pt, color= black!80, fill= yellow] (1,.25+.25/4)--(1,.25+.25/4+.25/4)--(1+.25/2,.25+.25/4+.25/4)--(1+.25/2,.25+.25/4)--(1,.25+.25/4);
\draw[line width=.9pt, color= black!80, fill= yellow] (1+.25/2,.25)--(1+.25/2,.25+.25/4)--(1+.25/2+.25/2,.25+.25/4)--(1+.25/2+.25/2,.25)--(1+.25/2,.25);
\draw[line width=.9pt, color= black!80, fill= orange] (1,.25)--(1,.25+.25/4)--(1+.25/2,.25+.25/4)--(1+.25/2,.25)--(1,.25);
\draw[line width=.9pt, color= black!80, fill= orange] (1+.25/2,.25+.25/4)--(1+.25/2,.25+.25/4+.25/4)--(1+.25/2+.25/2,.25+.25/4+.25/4)--(1+.25/2+.25/2,.25+.25/4)--(1+.25/2,.25+.25/4);
	\begin{scope}[shift={(2.8,.5)}]
    \draw
		plot[mark=square*, mark options={fill=orange}] (0.25,0)
		node[right]{\quad\scalebox{.8}{-1}};%
    \draw[yshift=\baselineskip]
		plot[mark=square*, mark options={fill=yellow}] (0.25,0)
		node[right]{\quad\scalebox{.8}{1}};
 \draw[yshift=2\baselineskip]
		plot[mark=square*, mark options={fill=white}] (0.25,0)
		node[right]{\quad\scalebox{.8}{0}};
	\end{scope}
\draw[line width=.5pt, color=gray] (0,0)--(2,0)--(2,2)--(0,2)--(0,0); 
\draw[line width=.5pt, color=gray] (0,0)--(2,2); 
\draw (2.3,2.1)node[] {\scalebox{1}{$\displaystyle \Delta$}};
\end{tikzpicture}
\end{center}
\caption{Values of $\tfrac{k_{IJ}}{\sqrt{\abs{I}\abs{J}}}$}\label{fig:IvacioJ}
\end{figure}
Notice that while $I\times J$ lies above the diagonal, $J\times I$ is contained in $\{y<x\}$. When $B(K)$ does not intersect $(I\times J)\cup(J\times I)$ then $\iint_{B(K)}k_{IJ}(x,y)dxdy=0$. Assume now that $B(K)\cap[(I\times J)\cup (J\times I)]\neq \emptyset$. Since $\iint_{Q}k_{IJ}(x,y)dxdy=0$, if we show that $B(K)\cap [(I\times J)\cup(J\times I)]\neq \emptyset$ implies $(I\times J)\cup(J\times I)\subseteq B(K)$ we have $\iint_{B(K)}k_{IJ}dxdy=0$. Since the set $B(K)\cap [(I\times J)\cup(J\times I)]=[(K^-\times K^+)\cup (K^+\times K^-)]\cap [(I\times J)\cup(J\times I)]$ is nonempty, we see that $(K^-\times K^+)\cap (I\times J)\neq \emptyset$. Since  $K^-\cap I\neq \emptyset$ and $K^+\cap J\neq \emptyset$ and $K$, $I$ and $J$ are dyadic intervals with $I\cap J=\emptyset$, we must have that $K^-\supset I$ and $K^+\supset J$. Therefore $B(K)\supset [(I\times J)\cup (J\times I)]$.

Let us assume now that $I$ and $J$ are nested. For example that $I\subsetneqq J$. Figure~\ref{fig:valueskIJ} depicts in this situation the normalized kernel $\tfrac{k_{IJ}}{\sqrt{\abs{I}\abs{J}}}$.


\begin{figure}[htp]
\begin{center}
\begin{tikzpicture}[scale=2]
\draw[line width=.9pt, color= gray!50] (.5,0)--(1,0); 
\draw[line width=.5pt, snake=brace] (-.08,.5)--(-.08,1);
\node at (-.2,.75) {\scalebox{1}{$J$}};
\draw[line width=.9pt, color= gray!50] (1/2+1/4,0)--(1/2+1/4+1/8,0); 
\draw[line width=.5pt, snake=brace] (1/2+1/4+1/8+1/64,-.08)--(1/2+1/4-1/64,-.08);
\node at (.8,-.25) {\scalebox{1}{$I$}};
\draw[line width=.9pt, color= gray!50] (.5,0)--(1,0)--(1,2)--(.5,2)--(.5,0); 
\draw[line width=.9pt, color= gray!50] (0,.5)--(0,1)--(2,1)--(2,.5)--(0,.5); 
\draw[line width=.9pt, color= gray!50] (.75,0)--(.875,0)--(.875,2)--(.75,2)--(.75,0); 
\draw[line width=.9pt, color= gray!50] (0,.75)--(0,.875)--(2,.875)--(2,.75)--(0,.75); 
%
\draw[line width=.9pt, color= black!80, fill= yellow] (.75+1/16,0)--(.875,0)--(.875,.5)--(.75+1/16,.5)--(.75+1/16,0);
\draw[line width=.9pt, color= black!80, fill= yellow] (.75+1/16,1)--(.875,1)--(.875,2)--(.75+1/16,2)--(.75+1/16,1);
\draw[line width=.9pt, color= black!80, fill= yellow] (0,.75+1/16)--(0,.875)--(.5,.875)--(.5,.75+1/16)--(0,.75+1/16);
\draw[line width=.9pt, color= black!80, fill= yellow] (1,.75+1/16)--(1,.875)--(2,.875)--(2,.75+1/16)--(1,.75+1/16);
\draw[line width=.9pt, color= black!80, fill= orange] (.75,0)--(.75+1/16,0)--(.75+1/16,.5)--(.75,.5)--(.75,0);
\draw[line width=.9pt, color= black!80, fill= orange] (.75,1)--(.75+1/16,1)--(.75+1/16,2)--(.75,2)--(.75,1);
\draw[line width=.9pt, color= black!80, fill= orange] (0,.75)--(0,.75+1/16)--(.5,.75+1/16)--(.5,.75)--(0,.75);
\draw[line width=.9pt, color= black!80, fill= orange] (1,.75)--(1,.75+1/16)--(2,.75+1/16)--(2,.75)--(1,.75);
\draw[line width=.9pt, color= black!80, fill= green] (.75+1/16,.5)--(.875,.5)--(.875,.75)--(.75+1/16,.75)--(.75+1/16,.5);
\draw[line width=.9pt, color= black!80, fill= green] (.5,.75+1/16)--(.5,.875)--(.75,.875)--(.75,.75+1/16)--(.5,.75+1/16);
\draw[line width=.9pt, color= black!80, fill= magenta] (.75,.5)--(.75+1/16,.5)--(.75+1/16,.75)--(.75,.75)--(.75,.5);
\draw[line width=.9pt, color= black!80, fill= magenta] (.5,.75)--(.75,.75)--(.75,.75+1/16)--(.5,.75+1/16)--(.5,.75);
%
	\begin{scope}[shift={(2.8,.25)}]
	\draw
		plot[mark=square*, mark options={fill=magenta}] (0.25,0)
		node[right]{\quad\scalebox{.8}{-2}};
	\draw[yshift=\baselineskip]
		plot[mark=square*, mark options={fill=green}] (0.25,0)
		node[right]{\quad\scalebox{.8}{2}};
    \draw[yshift=2\baselineskip]
		plot[mark=square*, mark options={fill=orange}] (0.25,0)
		node[right]{\quad\scalebox{.8}{-1}};%
    \draw[yshift=3\baselineskip]
		plot[mark=square*, mark options={fill=yellow}] (0.25,0)
		node[right]{\quad\scalebox{.8}{1}};
 \draw[yshift=4\baselineskip]
		plot[mark=square*, mark options={fill=white}] (0.25,0)
		node[right]{\quad\scalebox{.8}{0}};
	\end{scope}
\draw[line width=.5pt, color=gray] (0,0)--(2,0)--(2,2)--(0,2)--(0,0); 
\draw[line width=.5pt, color=gray] (0,0)--(2,2); 
\draw (2.3,2.1)node[] {\scalebox{1}{$\displaystyle \Delta$}};
\end{tikzpicture}
\end{center}
\caption{Values of $\tfrac{k_{IJ}}{\sqrt{\abs{I}\abs{J}}}$}\label{fig:valueskIJ}
\end{figure}

Since $k_{IJ}(x,y)=k_{IJ}(y,x)$ and $B(K)$ is symmetric, we only need to show that  $\iint_{K^+\times K^-}k_{IJ}(x,y)dxdy=0$. When $j(K)\geq j(J)+1$, $\supp k_{IJ}\cap B(K)=\emptyset$ and $\iint_{B(K)}k_{IJ}(x,y)dxdy=0$. Assume on the other hand that $0\leq j(K)\leq j(J)$. In this case the intersection of the support of $k_{IJ}$ and $B(K)$ can still be empty or, if not, the kernel $k_{IJ}(x,y)$ on $K^+\times K^-$ takes only two opposite constant non trivial values on subsets of the same area. Hence, again, $\iint_{B(K)}k_{IJ}(x,y)dxdy=0$. See Figure~\ref{fig:positionsK} where
two possible positions of $K$ when $I\varsubsetneqq J$ are illustrated.


\begin{figure}[h]
\begin{tikzpicture}[scale=2]
\begin{minipage}{0.5\textwidth}
\draw[line width=.9pt, color= gray!50] (.5,0)--(1,0);
\draw[line width=.9pt, color= gray!50] (1/2+1/4,0)--(1/2+1/4+1/8,0);
\draw[line width=.9pt, color= gray!50] (.5,0)--(1,0)--(1,2)--(.5,2)--(.5,0);
\draw[line width=.9pt, color= gray!50] (0,.5)--(0,1)--(2,1)--(2,.5)--(0,.5);
\draw[line width=.9pt, color= gray!50] (.75,0)--(.875,0)--(.875,2)--(.75,2)--(.75,0);
\draw[line width=.9pt, color= gray!50] (0,.75)--(0,.875)--(2,.875)--(2,.75)--(0,.75);
%
\draw[line width=.9pt, color= black!80, fill= yellow] (.75+1/16,0)--(.875,0)--(.875,.5)--(.75+1/16,.5)--(.75+1/16,0);
\draw[line width=.9pt, color= black!80, fill= yellow] (.75+1/16,1)--(.875,1)--(.875,2)--(.75+1/16,2)--(.75+1/16,1);
\draw[line width=.9pt, color= black!80, fill= yellow] (0,.75+1/16)--(0,.875)--(.5,.875)--(.5,.75+1/16)--(0,.75+1/16);
\draw[line width=.9pt, color= black!80, fill= yellow] (1,.75+1/16)--(1,.875)--(2,.875)--(2,.75+1/16)--(1,.75+1/16);
\draw[line width=.9pt, color= black!80, fill= orange] (.75,0)--(.75+1/16,0)--(.75+1/16,.5)--(.75,.5)--(.75,0);
\draw[line width=.9pt, color= black!80, fill= orange] (.75,1)--(.75+1/16,1)--(.75+1/16,2)--(.75,2)--(.75,1);
\draw[line width=.9pt, color= black!80, fill= orange] (0,.75)--(0,.75+1/16)--(.5,.75+1/16)--(.5,.75)--(0,.75);
\draw[line width=.9pt, color= black!80, fill= orange] (1,.75)--(1,.75+1/16)--(2,.75+1/16)--(2,.75)--(1,.75);
\draw[line width=.9pt, color= black!80, fill= green] (.75+1/16,.5)--(.875,.5)--(.875,.75)--(.75+1/16,.75)--(.75+1/16,.5);
\draw[line width=.9pt, color= black!80, fill= green] (.5,.75+1/16)--(.5,.875)--(.75,.875)--(.75,.75+1/16)--(.5,.75+1/16);
\draw[line width=.9pt, color= black!80, fill= magenta] (.75,.5)--(.75+1/16,.5)--(.75+1/16,.75)--(.75,.75)--(.75,.5);
\draw[line width=.9pt, color= black!80, fill= magenta] (.5,.75)--(.75,.75)--(.75,.75+1/16)--(.5,.75+1/16)--(.5,.75);
%
%
\draw[line width=1.2pt, fill=lightgray, fill opacity=0.5] (.75,.5)--(1,.5)--(1,.75)--(.75,.75)--(.75,.5);
\draw[line width=1.2pt, fill=lightgray, fill opacity=0.5] (.5,.75)--(.75,.75)--(.75,1)--(.5,1)--(.5,.75);
\draw[line width=.5pt, color=gray] (0,0)--(2,0)--(2,2)--(0,2)--(0,0); 
\draw[line width=.5pt, color=gray] (0,0)--(2,2); 
\draw[color=white] (2.3,2.1)node[] {\scalebox{1}{$\displaystyle \Delta$}};
\end{minipage}
\end{tikzpicture}
\begin{tikzpicture}[scale=2]
\begin{minipage}{0.5\textwidth}
\draw[line width=.9pt, color= gray!50] (.5,0)--(1,0);
\draw[line width=.9pt, color= gray!50] (1/2+1/4,0)--(1/2+1/4+1/8,0);
\draw[line width=.9pt, color= gray!50] (.5,0)--(1,0)--(1,2)--(.5,2)--(.5,0);
\draw[line width=.9pt, color= gray!50] (0,.5)--(0,1)--(2,1)--(2,.5)--(0,.5);
\draw[line width=.9pt, color= gray!50] (.75,0)--(.875,0)--(.875,2)--(.75,2)--(.75,0);
\draw[line width=.9pt, color= gray!50] (0,.75)--(0,.875)--(2,.875)--(2,.75)--(0,.75);
%
\draw[line width=.9pt, color= black!80, fill= yellow] (.75+1/16,0)--(.875,0)--(.875,.5)--(.75+1/16,.5)--(.75+1/16,0);
\draw[line width=.9pt, color= black!80, fill= yellow] (.75+1/16,1)--(.875,1)--(.875,2)--(.75+1/16,2)--(.75+1/16,1);
\draw[line width=.9pt, color= black!80, fill= yellow] (0,.75+1/16)--(0,.875)--(.5,.875)--(.5,.75+1/16)--(0,.75+1/16);
\draw[line width=.9pt, color= black!80, fill= yellow] (1,.75+1/16)--(1,.875)--(2,.875)--(2,.75+1/16)--(1,.75+1/16);
\draw[line width=.9pt, color= black!80, fill= orange] (.75,0)--(.75+1/16,0)--(.75+1/16,.5)--(.75,.5)--(.75,0);
\draw[line width=.9pt, color= black!80, fill= orange] (.75,1)--(.75+1/16,1)--(.75+1/16,2)--(.75,2)--(.75,1);
\draw[line width=.9pt, color= black!80, fill= orange] (0,.75)--(0,.75+1/16)--(.5,.75+1/16)--(.5,.75)--(0,.75);
\draw[line width=.9pt, color= black!80, fill= orange] (1,.75)--(1,.75+1/16)--(2,.75+1/16)--(2,.75)--(1,.75);
\draw[line width=.9pt, color= black!80, fill= green] (.75+1/16,.5)--(.875,.5)--(.875,.75)--(.75+1/16,.75)--(.75+1/16,.5);
\draw[line width=.9pt, color= black!80, fill= green] (.5,.75+1/16)--(.5,.875)--(.75,.875)--(.75,.75+1/16)--(.5,.75+1/16);
\draw[line width=.9pt, color= black!80, fill= magenta,] (.75,.5)--(.75+1/16,.5)--(.75+1/16,.75)--(.75,.75)--(.75,.5);
\draw[line width=.9pt, color= black!80, fill= magenta] (.5,.75)--(.75,.75)--(.75,.75+1/16)--(.5,.75+1/16)--(.5,.75);
%
%
\draw[line width=1.2pt, fill=lightgray, fill opacity=0.5] (.5,0)--(1,0)--(1,.5)--(.5,.5)--(.5,0);
\draw[line width=1.2pt, fill=lightgray, fill opacity=0.5] (0,.5)--(.5,.5)--(.5,1)--(0,1)--(0,.5);
\draw[line width=.5pt, color=gray] (0,0)--(2,0)--(2,2)--(0,2)--(0,0); 
\draw[line width=.5pt, color=gray] (0,0)--(2,2); 
\draw[color=white] (2.3,2.1)node[] {\scalebox{1}{$\displaystyle \Delta$}};
\end{minipage}
\end{tikzpicture}
\caption{On the left, $K$ equals $J$, and on the right,  $K$ is the father of $J$.}\label{fig:positionsK}
\end{figure}
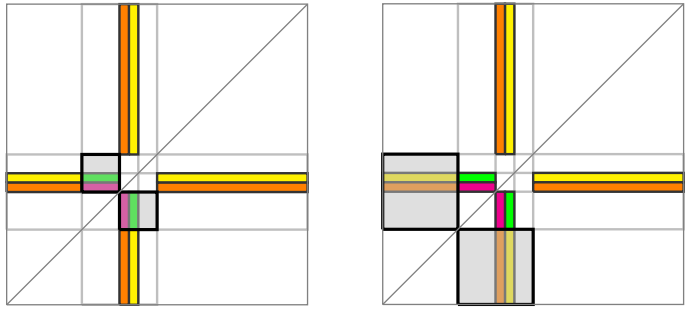

\smallskip
{\it Proof of (\ref{lem:geometricproperties}.\ref{item:diez}).\,}
Let us start by computing $\mathcal{I}(j,I,I)$ for $j\geq 0$ and $I\in\Dy^+$. From \textit{(\ref{lem:geometricpropertiesdelta}.\ref{item:tres})} we get
\begin{align*}
\mathcal{I}(j,I,I) &= \iint_{\Lambda_j}[h_I(x)-h_I(y)]^2 dxdy\\
&= \abs{I}^{-1}\sum_{J\in\Dy^j}\iint_{B(J)}[4\mathcal{X}_{B(I)}+\mathcal{X}_{C(I)}] dx dy\\
&= 4 \abs{I}^{-1}\sum_{J\in\Dy^j} m(B(J)\cap B(I))+ \abs{I}^{-1}\sum_{J\in\Dy^j} m(B(J)\cap C(I))
\end{align*}
for $j\geq 0$ and $I\in\Dy^+$. Hence, since from \textit{(\ref{lem:geometricpropertiesdelta}.\ref{item:uno})} and \textit{(\ref{lem:geometricpropertiesdelta}.\ref{item:tres})}, $B(J)\cap B(I)=\emptyset$ for $I\neq J$ and then applying  \textit{(\ref{lem:geometricproperties}.\ref{item:ocho})}
\begin{align*}
\sum_{j\geq 0}& 2^{j(1+2\lambda)}\mathcal{I}(j,I,I) \\
&= \sum_{j\geq 0}2^{j(1+2\lambda)}\left\{4 \abs{I}^{-1}\sum_{J\in\Dy^j} m(B(J)\cap B(I))+ \abs{I}^{-1} \sum_{J\in\Dy^j} m(B(J)\cap C(I))\right\}\\
&=4 \abs{I}^{-1} 2^{j(I)(1+2\lambda)}\frac{\abs{I}^2}{2}+ c_{\lambda}(\abs{I}^{-2\lambda}-1)\\
&= (2+c_{\lambda})\abs{I}^{-2\lambda}-c_{\lambda}.
\end{align*}
\end{proof}

For $0<\lambda<1$, a function $f\in L^2$ is said to belong to the \textbf{Besov space} $B^{\lambda}_{2,dy}$ if the function $\tfrac{f(x)-f(y)}{\delta(x,y)^\lambda}$ belongs to $L^2(Q,\tfrac{dx dy}{\delta(x,y)})$. In other words, $f\in B^{\lambda}_{2,dy}$ if and only if
\begin{equation*}
\norm{f}^2_{B^{\lambda}_{2,dy}}=\norm{f}^2_{L^2}+\iint_{Q}\frac{\abs{f(x)-f(y)}^2}{\delta(x,y)^{1+2\lambda}}dx dy
\end{equation*}
is finite.

For our purposes the main result concerning $B^{\lambda}_{2,dy}$ is the following Haar wavelet characterization of the Besov space. For the classical nondyadic Euclidean case see for example \cite{Meyer92waveletsI}.

\begin{theorem}\label{thm:equivalencianormas}
Let $0<\lambda<1$ be given. The space $B^{\lambda}_{2,dy}$ coincides with the set of all square integrable functions on $\mathbb{R}^+$ for which
\begin{equation*}
\sum_{I\in\Dy^+}\abs{\frac{\proin{f}{h_I}}{\abs{I}^{\lambda}}}^2<\infty.
\end{equation*}
Moreover, $\norm{f}_{L^2}+\left(\sum_{I\in\Dy^+}\abs{\frac{\proin{f}{h_I}}{\abs{I}^{\lambda}}}^2\right)^{\tfrac{1}{2}}$ is equivalent to $\norm{f}_{B^{\lambda}_{2,dy}}$.
\end{theorem}

\begin{proof}
We start by noticing that, from the definition of $Q$ as a union of the squares $(k-1,k)^2$, $k\in\mathbb{Z}^+$, there is no interference between blocks corresponding to different values of $k$ and then it is enough to prove that
$\norm{f}_{L^2(0,1)}+\left(\sum_{I\in\Dy^+_{(0,1)}}\left(\frac{\abs{\proin{f}{h_I}}}{\abs{I}^{\lambda}}\right)^2\right)^{\tfrac{1}{2}}$
is equivalent to $\norm{f}_{L^2(0,1)}+\left(\iint_{(0,1)^2}\frac{\abs{f(x)-f(y)}^2}{\delta(x,y)^{1+2\lambda}}\right)^{\tfrac{1}{2}}$, with $\Dy^+_{(0,1)}=\set{I\in\Dy^+: I\subset (0,1)}$.

Assume then that $f$ is an $L^2(0,1)$ function such that $\sum_{I\in\Dy^+_{(0,1)}}\frac{\abs{\proin{f}{h_I}}^2}{\abs{I}^{2\lambda}}<\infty$. Let $\mathcal{F}_n$ be an increasing sequence of finite subfamilies of $\Dy^+_{(0,1)}$ with $\cup_{n=1}^{\infty}\mathcal{F}_n=\Dy^+_{(0,1)}$ and  if $f_n=\sum_{I\in\mathcal{F}_n}\proin{f}{h_I}h_I$ we have both the $L^2(0,1)$ and a.e. pointwise convergence of $f_n$ to $f$. Then from Fatou's Lemma we have that
\begin{align*}
\iint_{(0,1)^2}\frac{\abs{f(x)-f(y)}^2}{\delta(x,y)^{1+2\lambda}}dxdy
&=\iint_{(0,1)^2}\lim_{n\to\infty}\frac{\abs{f_n(x)-f_n(y)}^2}{\delta(x,y)^{1+2\lambda}}dx dy\\
&\leq \liminf_{n\to\infty}\iint_{(0,1)^2}\frac{\abs{f_n(x)-f_n(y)}^2}{\delta(x,y)^{1+2\lambda}}dx dy.
\end{align*}
Now, since each $f_n\in\mathcal{S}(\mathscr{H})$, from Theorem \ref{thm:besovdyaigualhaar} we get
\begin{align*}
\iint_{(0,1)^2}\frac{\abs{f_n(x)-f_n(y)}^2}{\delta(x,y)^{1+2\lambda}}dx dy
&=\sum_{I\in\mathcal{F}_n}\abs{\proin{f}{h_I}}^2[(2+c_{\lambda})\abs{I}^{-2\lambda}-c_{\lambda}]\\
&\leq 2 \sum_{I\in\Dy^+_{(0,1)}}\frac{\abs{\proin{f}{h_I}}^2}{\abs{I}^{2\lambda}},
\end{align*}
hence
\begin{equation*}
\iint_{(0,1)^2}\frac{\abs{f(x)-f(y)}^2}{\delta(x,y)^{1+2\lambda}}dx dy\leq 2\sum_{I\in\Dy^+_{(0,1)}}\frac{\abs{\proin{f}{h_I}}^2}{\abs{I}^{2\lambda}}.
\end{equation*}

In order to prove the opposite inequality let us start by noticing that the identity \eqref{eq:besovdyaigualhaar} in Theorem~\ref{thm:besovdyaigualhaar} provides, by polarization, the following formula which holds for every $\varphi$ and $\psi\in \mathcal{S}(\mathscr{H})$
\begin{multline}\label{eq:polarization}
\iint_{(0,1)^2}\frac{\varphi(x)-\varphi(y)}{\delta(x,y)^\lambda}\frac{\psi(x)-\psi(y)}{\delta(x,y)^\lambda}\frac{dx dy}{\delta(x,y)}
\\=
\sum_{I\in\Dy^+_{(0,1)}}\proin{\varphi}{h_I}\proin{\psi}{h_I}[(2+c_{\lambda})\abs{I}^{-2\lambda}-c_{\lambda}].
\end{multline}

Assume that $f\in B^{\lambda}_{2,dy}$. Since for any $\psi\in\mathcal{S}(\mathscr{H})$ by  \textit{(\ref{lem:geometricpropertiesdelta}.\ref{item:cuatro})}, the function $\frac{\psi(x)-\psi(y)}{\delta(x,y)^{1+2\lambda}}$ has support at a positive $\delta$-distance of the diagonal $\Delta$, we have that it is bounded on $(0,1)^2$. Hence $\frac{\psi(x)-\psi(y)}{\delta(x,y)^{1+2\lambda}}\in L^2((0,1)^2,dxdy)$.  Taking in \eqref{eq:polarization} $f_n=\sum_{I\in\mathcal{F}_n}\proin{f}{h_I}h_I$ instead of $\varphi$ with $\mathcal{F}_n$ as before, we get
\begin{equation*}
\iint\limits_{(0,1)^2}\frac{f_n(x)-f_n(y)}{\delta(x,y)^\lambda}\frac{\psi(x)-\psi(y)}{\delta(x,y)^\lambda}\frac{dx dy}{\delta(x,y)}=
\sum_{\substack{I\in\mathcal{F}_n, \\ \proin{\psi}{h_I}\neq 0}}
\proin{f}{h_I}\proin{\psi}{h_I}[(2+c_{\lambda})\abs{I}^{-2\lambda}-c_{\lambda}].
\end{equation*}
Now since $f_n(x)-f_n(y)$ tends $f(x)-f(y)$ in $L^2((0,1)^2,dx dy)$ and $\psi\in\mathcal{S}(\mathscr{H})$ we get
\begin{equation*}
\iint\limits_{(0,1)^2}\frac{f(x)-f(y)}{\delta(x,y)^\lambda}\frac{\psi(x)-\psi(y)}{\delta(x,y)^\lambda}\frac{dx dy}{\delta(x,y)}=
\sum_{I\in\Dy^+_{(0,1)}}\proin{f}{h_I}\proin{\psi}{h_I}[(2+c_{\lambda})\abs{I}^{-2\lambda}-c_{\lambda}].
\end{equation*}
We have to prove that $\sum\frac{\abs{\proin{f}{h_I}}^2}{\abs{I}^{2\lambda}}$ is finite. This quantity can be estimated by duality, since
\begin{equation*}
\left(\sum_{I\in\Dy^+_{(0,1)}}\frac{\abs{\proin{f}{h_I}}^2}{\abs{I}^{2\lambda}}\right)^{\tfrac{1}{2}}=\sup \sum_{I\in\Dy^+_{(0,1)}}\frac{\proin{f}{h_I}}{\abs{I}^{\lambda}}b_I
\end{equation*}
where the supremum is taken on the family of all sequences $(b_I)$ with $\sum_{I\in\Dy^+_{(0,1)}} b_I^2\leq 1$ and $b_I=0$ except for  a finite  number of $I$'s in $\Dy^+_{(0,1)}$. Notice that every such sequence $(b_I)$ can be uniquely determined by the sequence of Haar coefficient of the function $\psi= \sum b_I \abs{I}^{-\lambda}[(2+c_{\lambda})\abs{I}^{-2\lambda}-c_{\lambda}]^{-1}h_I\in \mathcal{S}(\mathscr{H})$. In fact, $b_I=\proin{\psi}{h_I}\abs{I}^{\lambda}[(2+c_{\lambda})\abs{I}^{-2\lambda}-c_{\lambda}]$. Hence the condition $\sum_{I\in\Dy^+_{(0,1)}} b_I^2\leq 1$ becomes $\sum_{I\in\Dy^+_{(0,1)}}\proin{\psi}{h_I}^2\abs{I}^{2\lambda}[(2+c_{\lambda})\abs{I}^{-2\lambda}-c_{\lambda}]^2\leq 1$. So
\begin{align*}
\sum_{I\in\Dy^+_{(0,1)}}\frac{\proin{f}{h_I}}{\abs{I}^{\lambda}}b_I
&=\sum_{I\in\Dy^+_{(0,1)}}\proin{f}{h_I}\proin{\psi}{h_I}[(2+c_{\lambda})\abs{I}^{-2\lambda}-c_{\lambda}]\\
&=\iint\limits_{(0,1)^2}\frac{f(x)-f(y)}{\delta(x,y)^\lambda}\frac{\psi(x)-\psi(y)}{\delta(x,y)^\lambda}\frac{dx dy}{\delta(x,y)}\\
&\leq \left[\iint\limits_{(0,1)^2}\left(\frac{\abs{f(x)-f(y)}}{\delta (x,y)^{\lambda}} \right)^2\frac{dx dy}{\delta (x,y)}\right]^{\tfrac{1}{2}}
\iint\limits_{(0,1)^2}\left(\frac{\abs{\psi(x)-\psi(y)}}{\delta (x,y)^{\lambda}}\right)^2\frac{dx dy}{\delta(x,y)}\\
&\leq \norm{f}_{B^{\lambda}_{2,dy}},
\end{align*}
since $\iint_{(0,1)^2}\left(\frac{\abs{\psi(x)-\psi(y)}}{\delta (x,y)^{\lambda}}\right)^2\frac{dx dy}{\delta(x,y)}=\sum_{I\in\Dy^+_{(0,1)}}\proin{\psi}{h_I}^2\abs{I}^{2\lambda}[(2+c_{\lambda})\abs{I}^{-2\lambda}-c_{\lambda}]^2\leq 1$ for $\psi\in \mathcal{S}(\mathscr{H})$.
\end{proof}

As a corollary of Theorem~\ref{thm:equivalencianormas} we easily obtain the following density result.
\begin{corollary}\label{coro:convergenciaBesovdyadic}
For $f\in B^{\lambda}_{2,dy}$ with $P_0f=0$ and $f_n=\sum_{I\in\mathcal{F}_n}\proin{f}{h_I}h_I$ with $\mathcal{F}_n\subset\mathcal{F}_{n+1}$, $\mathcal{F}_n$ finite and $\cup_{n=1}^{\infty}\mathcal{F}_n=\Dy^+$, we have $f_n\to f$ in $B^{\lambda}_{2,dy}$ as $n\to\infty$.
\end{corollary}

The above result allows to extend Theorem~\ref{thm:integralformofoperatorinlinearspan} to dyadic Besov functions with vanishing means between integers.

\begin{theorem}\label{thm:derivadaBetaparafuncionenBesov}
Let $0<\beta<\lambda<1$ be given. Then for each $f\in B^{\lambda}_{2,dy}$ with $P_0f=0$, we have
\begin{equation}\label{eq:characterizationgeneralcase}
\sum_{I\in\Dy^+}\abs{I}^{-\beta}\proin{f}{h_I}h_I(x)=\frac{2^{\beta}-1}{2^\beta}\int_{\mathbb{R}^+}
\frac{f(x)-f(y)}{\delta(x,y)^{1+\beta}}dy
\end{equation}
as functions in $L^2$.
\end{theorem}
\begin{proof}
For  $f_n$ as in Corollary \ref{coro:convergenciaBesovdyadic}, Theorem \ref{thm:integralformofoperatorinlinearspan} provides the identity \eqref{eq:characterizationgeneralcase}. Hence to prove \eqref{eq:characterizationgeneralcase} in our new situation, it suffices to prove that both sides in \eqref{eq:characterizationgeneralcase} define bounded operators with respect to the norms $B^{\lambda}_{2,dy}$ in the domain and $L^2$ in its image. For the left hand side, we see that
\begin{equation*}
\norm{\sum_{I\in\Dy^+}\abs{I}^{-\beta}\proin{f}{h_I}h_I}^2_2=\sum_{I\in\Dy^+}\abs{I}^{2(\lambda-\beta)}
\left(\frac{\abs{\proin{f}{h_I}}}{\abs{I}^{\lambda}}\right)^2
\end{equation*}
which is bounded by the $B^{\lambda}_{2,dy}$ norm of $f$ from Theorem \ref{thm:equivalencianormas}. For the operator on the right hand side of \eqref{eq:characterizationgeneralcase} we start by splitting the integral in the following way
\begin{equation*}
\int_{\mathbb{R}^+}
\frac{f(x)-f(y)}{\delta(x,y)^{1+\beta}}dy=\int_{\delta(x,y)<2}
\frac{f(x)-f(y)}{\delta(x,y)^{1+\beta}}dy + \int_{\delta(x,y)\geq 2}
\frac{f(x)-f(y)}{\delta(x,y)^{1+\beta}}dy.
\end{equation*}
Applying \textit{(\ref{lem:geometricpropertiesdelta}.\ref{item:cinco})} in Lemma~\ref{lem:geometricpropertiesdelta} we obtain that the $L^2$ norm of the first term in the right is bounded by
\begin{align*}
\int_{\mathbb{R}^+}&\left(\int_{[x]}^{[x]+1}\frac{\abs{f(x)-f(y)}^2}{\delta(x,y)^{1+2\lambda}}dy\right)
 \left(\int_{[x]}^{[x]+1}\frac{dy}{\delta(x,y)^{1-2(\lambda-\beta)}}\right)dx\\
 &\leq C\iint_{Q}
\frac{\abs{f(x)-f(y)}^2}{\delta(x,y)^{1+2\lambda}}dydx,
\end{align*}
as desired. On the other hand the square of $L^2$ norm of the second term is bounded by
\begin{align*}
\int_{\mathbb{R}^+}&\left(\int_{\delta(x,y)\geq 2}\frac{\abs{f(x)}+\abs{f(y)}}{\delta(x,y)^{\tfrac{1+\beta}{2}}}\frac{dy}{\delta(x,y)^{\tfrac{1+\beta}{2}}}\right)^2 dx\\
&\leq \int_{\mathbb{R}^+}\left(\int_{\delta(x,y)\geq 2}\frac{(\abs{f(x)}+\abs{f(y)})^2}{\delta(x,y)^{1+\beta}} dy\right)\left(\int_{\delta(x,y)\geq 2}\frac{dy}{\delta(x,y)^{1+\beta}}\right)dx\\
&\leq C\int_{\mathbb{R}^+}\int_{\delta(x,y)\geq 2}\frac{\abs{f(x)}^2}{\delta(x,y)^{1+\beta}}dy dx
\leq \overline{C}\norm{f}^2_{L^2}\leq \overline{C}\norm{f}^2_{B^{\lambda}_{2,dy}}.
\end{align*}
\end{proof}

\section{The main result}\label{sec:mainresultproof}
 In this section we state and prove a detailed formulation of Theorem~\ref{thm:consequencemainresult}.
With the operator $D^{\beta}$ and the spaces $B^{\lambda}_{2,dy}$ introduced in Section~\ref{sec:dyadicBesovspaces} the problem can now be formally written in the following way

\begin{equation*}
(P)\quad  \begin{dcases}
i \frac{\partial u}{\partial t} = D^{\beta}u  & \mbox{in $\mathbb{R}^+\times \mathbb{R}^+$ }\\
\medskip
 u(0)=u^0  & \mbox{in $\mathbb{R}^+$.}
\end{dcases}
\end{equation*}

\begin{theorem}\label{thm:mainresult}
For $0<\beta<\lambda<1$ and $u^0\in B^{\lambda}_2$ with $P_0u^0=0$, define
\begin{equation}\label{eq:seriessolucion}
u(t)=-\sum_{I\in\Dy}e^{i t \abs{I}^{-\beta}}\proin{u^0}{h_I}h_I, \quad t\geq 0.
\end{equation}
Then,
\begin{enumerate}[(\ref{thm:mainresult}.a)]
\item $u$ is continuous as a function of $t\in [0,\infty)$ with values in $B^{\lambda}_{2,dy}$ and $u(0)=u^0$. In other words, $\norm{u(t)-u(s)}_{B^{\lambda}_{2,dy}}\to 0$ for $s\to t$ and $t\geq 0$; \label{item:mainresultuno}
\item $u$ is differentiable as a function of $t\in (0,\infty)$ with respect to the norm $\norm{\cdot}_{B^{\lambda-\beta}_{2,dy}}$, and $\tfrac{du}{dt}=- iD^{\beta}u$. Precisely, $\norm{\frac{u(t+h)-u(t)}{h}+ i D^{\beta}u}_{B^{\lambda-\beta}_{2,dy}}\to 0$ when $h\to 0$;\label{item:mainresultdos}
\item there exists $Z\subset \mathbb{R}^+$ with  $\abs{Z}=0$ such that the series \eqref{eq:seriessolucion} defining $u(t)$ converges pointwise for every $t\in [0,1)$ outside $Z$;\label{item:mainresultcuatro}
\item $u(t)\to u^0$  pointwise almost everywhere on $\mathbb{R}^+$ when $t\to 0$.\label{item:mainresultcinco}
\end{enumerate}
\end{theorem}

Notice that pointwise convergence is not a consequence of convergence in the $B^{\lambda}_{2,dy}$ norm. In fact, with the standard notation for the Haar system $h^j_k(x)=2^{\tfrac{j}{2}}h(2^j x-k)$, we define a sequence of functions supported in $(0,1)$ in the following way. Let $n$ be a given positive integer. Then there exists one and only one $j=0,1,2,\ldots$ such that $2^j\leq n <2^{j+1}$. Set $f_n=2^{-\tfrac{j}{2}}h^j_{n-2^j}$. Then $\norm{f_n}_{L^2}=2^{-\tfrac{j}{2}}$ which tends to $0$ as $n\to\infty$. Since $D^{\lambda}f_n=2^{-\tfrac{j}{2}}D^{\lambda}h^j_{n-2^j}=2^{j\lambda}2^{-\tfrac{j}{2}}h^j_{n-2^j}=2^{-j(\tfrac{1}{2}-\lambda)}h^j_{n-2^j}$, we see that for $0<\lambda<\tfrac{1}{2}$, $\norm{D^{\lambda}f_n}_{L^2}\to 0$ as $n\to\infty$. Hence $f_n\to 0$ in the $B^{\lambda}_{2,dy}$ sense. Nevertheless $f_n$ does not converge pointwise.

\medskip
Before proving Theorem~\ref{thm:mainresult} we shall obtain some basic maximal estimates involved in the proofs of \textit{(\ref{thm:mainresult}.\ref{item:mainresultcuatro})} and \textit{(\ref{thm:mainresult}.\ref{item:mainresultcinco})}. With $M_{dy}$ we denote the Hardy-Littlewood dyadic maximal operator given by
\begin{equation*}
M_{dy}f(x)=\sup \frac{1}{\abs{I}}\int_{I}\abs{f(y)}dy
\end{equation*}
where the supremum is taken on the family of all dyadic intervals $I\in\Dy$ for which $x\in I$. Calder\'on's sharp maximal operator of order $\lambda$ is defined by
\begin{equation*}
M^{\#}_{\lambda}f(x)=\sup_{J}\frac{1}{\abs{J}^{1+\lambda}}\int_{J}\abs{f(y)-f(x)}dy,
\end{equation*}
where the supremum is taken on the family of all subintervals (dyadic or not) $J$ of $\mathbb{R}^+$  such that $x\in J$. In \cite{DeSa84}, see Corollary 11.6, DeVore and Sharpley prove that the $L_p$ norm of $M^{\#}_{\lambda}f$ is bounded by the $B^{\lambda}_{p}$ norm of $f$. For our purposes the case $p=2$ is of particular interest,
\begin{equation}\label{eq:DeVShaacotacion}
\norm{M^{\#}_{\lambda}f}_{L^2}\leq A\norm{f}_{B^{\lambda}_{2}}.
\end{equation}

When dealing with
\textit{(\ref{thm:mainresult}.\ref{item:mainresultcuatro})} and \textit{(\ref{thm:mainresult}.\ref{item:mainresultcinco})} two maximal operators related to the series \eqref{eq:seriessolucion} are also relevant. For $t>0$ set
\begin{equation*}
S_{t}^{\displaystyle *}f(x)=\sup_{N\in \mathbb{N}}\abs{S^{N}_{t}f(x)},\textrm{ where }S^{N}_{t}f(x)=\sum_{j=0}^{N}\sum_{k\in\mathbb{Z}^+}e^{i t 2^{j\beta}}\proin{f}{h^j_k}h^j_k(x).
\end{equation*}
Set
\begin{equation*}
S^{\displaystyle *} f(x)=\sup_{0<t<1}S_{t}^{\displaystyle *}f(x).
\end{equation*}
The next result contains the basic estimates of $S^{\displaystyle *}_t$ and $S^{\displaystyle *}$ in terms of $M_{dy}$ and $M^{\#}_{\lambda}$.
\begin{lemma}\label{lem:sobremainresult}
Let $f\in B^{\lambda}_2$ with $0<\beta<\lambda<1$ and $P_0f=0$. Then with $C:=2^{\lambda-\beta+1}(2^{\lambda-\beta}-1)$ we have
\begin{enumerate}[(\ref{lem:sobremainresult}.a)]
\item $S_{t}^{\displaystyle *}f(x)\leq Ct M^{\#}_{\lambda}f(x) + 2 M_{dy}f(x)$
for $t\geq 0$ and $x\in\mathbb{R}^+$;\label{item:sobremainresultuno}

\item $S^{\displaystyle *}f(x)\leq C M^{\#}_{\lambda}f(x) + 2 M_{dy}f(x)$ for $x\in\mathbb{R}^+$;\label{item:sobremainresultdos}

\item $\norm{S^{\displaystyle *}f}_{L^2}\leq (AC+2)\norm{f}_{B^{\lambda}_2}$, where $A$ is the constant in \eqref{eq:DeVShaacotacion}. \label{item:sobremainresulttres}
\end{enumerate}
\end{lemma}

\begin{proof}
For $f\in B^{\lambda}_2$, $t\geq 0$ and $N\in \mathbb{N}$, we have
\begin{equation}\label{eq:sumaryrestar}
\abs{S_{t}^{N}f(x)}\leq\abs{S_{t}^{N}f(x)-S_{0}^{N}f(x)}+\abs{S_{0}^{N}f(x)}.
\end{equation}
Since $S_{0}^{N}f(x)=P_Nf(x)$, where $P_N$ is the projection over the space $V_N$ of functions which are constant on each $I\in\Dy^{N}$, we have $\sup_N \abs{S_{0}^{N}f(x)}\leq  M_{dy}f(x)$. Let us now estimate the first term on the right hand side of \eqref{eq:sumaryrestar}.
For $x\in \mathbb{R}^+$ and $j\in \mathbb{N}$, let $k(x,j)\in\mathbb{Z}^+ $, be the only index for which $x\in I^{j}_{k(x,j)}$,
\begin{align}\label{eq:cauchyestimacionmaximalsharp}
\abs{S_{t}^{N}f(x)-S_{0}^{N}f(x)}
&\leq \abs{\sum_{j=0}^{N}\sum_{k\in\mathbb{Z}^+}(e^{i t 2^{j\beta}}-1)\proin{f}{h^j_k}h^j_k(x)}\notag\\
&= \abs{\sum_{j=0}^{N}(e^{i t 2^{j\beta}}-1)\!\!\left(\int_{I^j_{k(j,x)}}[f(y)-f(x)]h^j_{k(j,x)}(y)dy\right)h^j_{k(j,x)}(x)} \notag\\
&\leq \sum_{j=0}^{\infty}\abs{e^{i t 2^{j\beta}}-1}\frac{1}{\abs{I^j_{k(j,x)}}}\int_{I^j_{k(j,x)}}\abs{f(y)-f(x)}dy\notag\\
&= t \sum_{j=0}^{\infty}\frac{\abs{e^{i t 2^{j\beta}}-1}}{t 2^{j\lambda}}\frac{1}{\abs{I^j_{k(j,x)}}^{1+\lambda}}\int_{I^j_{k(j,x)}}\abs{f(y)-f(x)}dy\notag\\
&\leq 2 t \left(\sum_{j=0}^{\infty}2^{-(\lambda-\beta)j}\right)M^{\#}_{\lambda}f(x),
\end{align}
which proves \textit{(\ref{lem:sobremainresult}.\ref{item:sobremainresultuno})}. The estimate \textit{(\ref{lem:sobremainresult}.\ref{item:sobremainresultdos})} follows from \textit{(\ref{lem:sobremainresult}.\ref{item:sobremainresultuno})} by taking supremum for $t<1$. To show \textit{(\ref{lem:sobremainresult}.\ref{item:sobremainresulttres})} we invoke \eqref{eq:DeVShaacotacion}, and the $L^2$ boundedness  of the Hardy-Littlewood dyadic maximal operator.
\end{proof}

The next lemma gives the pointwise convergence of $S^{N}_{t}g(x)$ for every $x\in \mathbb{R}^+$ in a dense subspace of $B^{\lambda}_2$.

\begin{lemma}\label{lem:densoLipschitz}
Let $g$ be a Lipschitz function defined on $\mathbb{R}^+$. Then
\begin{equation*}
S^{N}_{t}g(x)= \sum_{j=0}^{N}\sum_{k\in\mathbb{Z}^+}e^{i t 2^{j\beta}}\proin{g}{h^j_k}h^j_k(x)
\end{equation*}
converges when $N\to\infty$, for every $x\in \mathbb{R}^+$ and every $t\geq 0$.
\end{lemma}

\begin{proof}
Fix $t\geq 0$ and $x\in \mathbb{R}^+$. We shall prove that $(S^{N}_{t} g(x): N=1,2,\ldots)$ is a Cauchy sequence of complex numbers. In fact, for $1\leq M\leq N$,
\begin{align*}
\abs{S^{N}_{t} g(x)-S^{M}_{t} g(x)}
&= \abs{\sum_{j=M+1}^{N}\sum_{k\in\mathbb{Z}^+}e^{i t 2^{j\beta}}\proin{g}{h^j_k}h^j_k(x)}\\
&= \abs{\sum_{j=M+1}^{N}\sum_{k\in\mathbb{Z}^+}e^{i t 2^{j\beta}}\left(\int_{\mathbb{R}^+}[g(y)-g(x)]h^j_k(y) dy\right)h^j_k(x)}\\
&\leq \sum_{j=M+1}^{N}\sum_{k\in\mathbb{Z}^+}\norm{g'}_{\infty} 2^j\int_{I^j_k}\abs{x-y} dy \mathcal{X}_{I^j_k}(x)\\
&=\norm{g'}_{\infty} \sum_{j=M+1}^{N}2^j\int_{I^j_{k(j,x)}}\abs{x-y} dy
\leq \norm{g'}_{\infty} \sum_{j=M+1}^{N}2^{-j}.
\end{align*}
\end{proof}

\begin{proof}[Proof of Theorem~\ref{thm:mainresult}]\quad

{\it Proof of (\ref{thm:mainresult}.\ref{item:mainresultuno}).\,}
From Theorem~\ref{thm:equivalencianormas} we see that for each $t>0$, $u(t)\in B^{\lambda}_{2,dy}$, since $u^0\in B^{\lambda}_2\subset B^{\lambda}_{2,dy}$. Moreover, for $t,s\geq 0$,
\begin{align*}
&\norm{u(t)-u(s)}_{B^{\lambda}_{2,dy}}
= \norm{\sum_{I\in\Dy^+} \left(e^{i t \abs{I}^{-\beta}}-e^{i s \abs{I}^{-\beta}}\right)\proin{u^0}{h_I}h_I}_{B^{\lambda}_{2,dy}}\\
&= \sum_{I\in\Dy^+}\abs{e^{i t \abs{I}^{-\beta}}-e^{i s \abs{I}^{-\beta}}}^2\abs{\proin{u^0}{h_I}}^2 + \sum_{I\in\Dy^+} \abs{e^{i t \abs{I}^{-\beta}}-e^{i s \abs{I}^{-\beta}}}^2 \frac{\proin{u^0}{h_I}^2}{\abs{I}^{2\lambda}}
\end{align*}
which tends to zero if $s\to t$.

{\it Proof of (\ref{thm:mainresult}.\ref{item:mainresultdos}).\,}
Let us prove that the formal derivative of $u(t)$ is actually the derivative in the sense of $B^{\lambda-\beta}_{2,dy}$. In fact, for $t>0$ and $h$ such that $t+h>0$
\begin{align*}
&\norm{\frac{u(t+h)-u(t)}{h}- i\sum_{I\in\Dy^+}e^{it\abs{I}^{-\beta}}\abs{I}^{-\beta}\proin{u^0}{h_I}h_I}^2_{B^{\lambda-\beta}_{2,dy}}\\
&=\norm{\sum_{I\in\Dy^+}e^{it\abs{I}^{-\beta}}\left[\frac{e^{ih\abs{I}^{-\beta}}-1}{h}-i\abs{I}^{-\beta}\right]\proin{u^0}{h_I}h_I}^2_{B^{\lambda-\beta}_{2,dy}}
\\
&\leq c \left\{\norm{\sum_{I\in\Dy^+}e^{it\abs{I}^{-\beta}}\left[\frac{e^{ih\abs{I}^{-\beta}}-1}{h}-i\abs{I}^{-\beta}\right]\proin{u^0}{h_I}h_I}^2_{L^2}\right.\\
&\phantom{\left\{\norm{\sum_{I\in\Dy^+}e^{it\abs{I}^{-\beta}}\left[\frac{e^{ih\abs{I}^{-\beta}}}{h}\right]}\right.}
\left.+\sum_{I\in\Dy^+}\abs{\frac{e^{ih\abs{I}^{-\beta}}-1}{h}-i\abs{I}^{-\beta}}^2\frac{\abs{\proin{u^0}{h_I}}^2}{\abs{I}^{2(\lambda-\beta)}}\right\}\\
&\leq c \sum_{I\in\Dy^+}\abs{I}^{2\beta}\abs{\frac{e^{ih\abs{I}^{-\beta}}-1}{h}-i\abs{I}^{-\beta}}^2\frac{\abs{\proin{u^0}{h_I}}^2}{\abs{I}^{2\lambda}}.
\end{align*}
Since, from Theorem~\ref{thm:equivalencianormas}, $\sum_{I\in\Dy^+}\frac{\abs{\proin{u^0}{h_I}}^2}{\abs{I}^{2\lambda}}<\infty$, and $\abs{I}^{2\beta}\bigl|\frac{e^{ih\abs{I}^{-\beta}}-1}{h}-i\abs{I}^{-\beta}\bigr|^2= \bigl|\frac{e^{ih\abs{I}^{-\beta}}-1}{\abs{I}^{-\beta}h}-i\bigr|^2\to 0$ as $h\to 0$ for each $I\in\Dy^+$, we obtain the result.

On the other hand since $u(t)\in B^{\lambda}_{2,dy}$ and since $\lambda>\beta$, $D^{\beta}u(t)$ is well defined and it is given by
\begin{equation*}
D^{\beta}u(t) = D^{\beta}\left(\sum_{I\in\Dy^+}e^{it\abs{I}^{-\beta}}\proin{u^0}{h_I}h_I\right)
= \sum_{I\in\Dy^+}e^{it\abs{I}^{-\beta}}\abs{I}^{-\beta}\proin{u^0}{h_I}h_I
= -i \frac{d u}{d t}.
\end{equation*}
Hence $u(t)$ is a solution of the nonlocal equation and \textit{(\ref{thm:mainresult}.\ref{item:mainresultdos})} is proved.

{\it Proof of (\ref{thm:mainresult}.\ref{item:mainresultcuatro}).\,}
The boundedness properties of $S^{\displaystyle *}_t$ and $S^{\displaystyle *}$ and the pointwise convergence on a dense subset of $B^{\lambda}_{2}$ allows us to use  standard arguments for the a.e. pointwise convergence of $S^{N}_{t} u^0$ for general $u^0\in B^{\lambda}_{2}$. We shall prove that the set $Z$ of all points $x$ in $\mathbb{R}^+$ such that for some $t\in (0,1)$
\begin{equation*}
\overline{L_t}(x):=\inf_{N}\sup_{n, m\geq N} \abs{S^{n}_{t} u^0(x)-S^{m}_{t} u^0(x)}>0
\end{equation*}
has measure zero. It is enough to show that for each $\varepsilon>0$, the Lebesgue measure of the set
$\{x\in \mathbb{R}^+: \overline{L_t}(x)>\varepsilon \textrm{ for some } t\in (0,1)\}$ vanishes.
Since, for any Lipschitz function $v$ defined on $\mathbb{R}^+$ and every $t\in (0,1)$,
\begin{equation*}
\abs{S^{n}_{t} u^0(x)-S^{m}_{t} u^0(x)}\leq \abs{S^{n}_{t} (u^0-v)(x)}+\abs{S^{n}_{t} v(x) - S^{m}_{t} v(x)}
 +\abs{S^{m}_{t} (v-u^0)(x)},
\end{equation*}
from Lemma \ref{lem:densoLipschitz}, we have $\overline{L_t}(x)\leq 2 S^{\displaystyle *}(u^0-v)(x)$.
So that, from \textit{(\ref{lem:sobremainresult}.\ref{item:sobremainresulttres})} we obtain
\begin{align*}
&\abs{\left\{x\in \mathbb{R}^+: \overline{L_t}(x)>\varepsilon \textrm{ for some } t\in (0,1)\right\}}
\leq \abs{\left\{x\in \mathbb{R}^+:  S^{\displaystyle *}(u^0-v)(x)>\frac{\varepsilon}{2} \right\}}\\
&\leq \frac{4}{\varepsilon^2}\norm{S^{\displaystyle *}(u^0-v)}^2_{L^2}
\leq \frac{4(AC +2)^2}{\varepsilon^2}\norm{u^0-v}^2_{B^{\lambda}_2}.
\end{align*}
Since $v$ is an arbitrary Lipschitz function in $\mathbb{R}^+$ we get that $\abs{Z}=0$. Hence for every $t\in [0,1)$ and every $x\notin Z$, $(S^{n}_{t}u^0(x): n=1,2,\ldots)$ is a Cauchy sequence which must converge to its $L^2$ limit, i.e. $u(t)(x)$ for $x\notin Z$ and $t\in [0,1)$.

{\it Proof of (\ref{thm:mainresult}.\ref{item:mainresultcinco}).\,}
For $x\notin Z$, taking the limit as $N\to\infty$ in \eqref{eq:cauchyestimacionmaximalsharp} we get the maximal estimate
\begin{equation*}
\sup_{t\in (0,1)}\frac{\abs{u(t)(x)-u^0(x)}}{t}\leq 2\frac{2^{\lambda-\beta}}{1-2^{-(\lambda-\beta)}}M^{\#}_{\lambda}u^0(x).
\end{equation*}
Since $M^{\#}_{\lambda}u^0$ belongs to $L^2$, the left hand side is finite almost everywhere, hence $u(t)(x)\to u^0(x)$ as $t\to 0$ almost everywhere.
\end{proof}

\bigskip

\subsection*{Acknowledgment}
The authors would like to acknowledge Marcelo Actis for his careful reading of the manuscript and fruitful discussions.


\def\cprime{$'$}
\providecommand{\bysame}{\leavevmode\hbox to3em{\hrulefill}\thinspace}
\providecommand{\MR}{\relax\ifhmode\unskip\space\fi MR }
\providecommand{\MRhref}[2]{%
  \href{http://www.ams.org/mathscinet-getitem?mr=#1}{#2}
}
\providecommand{\href}[2]{#2}

%
%
\bigskip
{\footnotesize
\textsc{Instituto de
Matem\'atica Aplicada del Litoral (IMAL)}

%
%

\textmd{G\"{u}emes 3450, S3000GLN Santa Fe, Argentina.}

\medskip

\begin{thebibliography}{10}

\bibitem{Carleson80}
Lennart Carleson, \emph{Some analytic problems related to statistical
  mechanics}, Euclidean harmonic analysis ({P}roc. {S}em., {U}niv. {M}aryland,
  {C}ollege {P}ark, {M}d., 1979), Lecture Notes in Math., vol. 779, Springer,
  Berlin, 1980, pp.~5--45. 

\bibitem{ChenGuo2007}
Jianqing Chen and Boling Guo, \emph{Strong instability of standing waves for a
  nonlocal {S}chr\"odinger equation}, Phys. D \textbf{227} (2007), no.~2,
  142--148. 

\bibitem{Cow83}
Michael~G. Cowling, \emph{Pointwise behavior of solutions to {S}chr\"odinger
  equations}, Harmonic analysis ({C}ortona, 1982), Lecture Notes in Math., vol.
  992, Springer, Berlin, 1983, pp.~83--90. 

\bibitem{DaKe82}
Bj{\"o}rn E.~J. Dahlberg and Carlos~E. Kenig, \emph{A note on the almost
  everywhere behavior of solutions to the {S}chr\"odinger equation}, Harmonic
  analysis ({M}inneapolis, {M}inn., 1981), Lecture Notes in Math., vol. 908,
  Springer, Berlin, 1982, pp.~205--209. 

\bibitem{DeSa84}
Ronald~A. DeVore and Robert~C. Sharpley, \emph{Maximal functions measuring
  smoothness}, Mem. Amer. Math. Soc. \textbf{47} (1984), no.~293, viii+115.

\bibitem{Dirac1928}
P.~A.~M. Dirac, \emph{The quantum theory of the electron}, Proc. R. Soc. Lond.
  A \textbf{117} (1928), 610--624.

\bibitem{KeRui83}
Carlos~E. Kenig and Alberto Ruiz, \emph{A strong type {$(2,\,2)$} estimate for
  a maximal operator associated to the {S}chr\"odinger equation}, Trans. Amer.
  Math. Soc. \textbf{280} (1983), no.~1, 239--246. 

\bibitem{Laskin02}
Nick Laskin, \emph{Fractional {S}chr\"odinger equation}, Phys. Rev. E
\textbf{66} (2002), no.~5, 056108, 7 pp. 

\bibitem{MaSkuKro2011}
F.~Maucher, S.~Skupin, and W.~Krolikowski, \emph{Collapse in the nonlocal
  nonlinear {S}chr\"odinger equation}, Nonlinearity \textbf{24} (2011), no.~7,
  1987--2001. 

\bibitem{Meyer92waveletsI}
Yves Meyer, \emph{Wavelets and operators}, Cambridge Studies in Advanced
  Mathematics, vol.~37, Cambridge University Press, Cambridge, 1992, Translated
  from the 1990 French original by D. H. Salinger. 

\bibitem{Muslih2010}
Sami~I. Muslih, Om~P. Agrawal, and Dumitru Baleanu, \emph{A fractional {D}irac
  equation and its solution}, J. Phys. A \textbf{43} (2010), no.~5, 055203, 13.

\bibitem{Peetre76}
Jaak Peetre, \emph{New thoughts on {B}esov spaces}, Duke University Mathematics
  Series, vol.~1, Mathematics Department, Duke University, Durham, N.C., 1976.

\bibitem{Sjo87}
Per Sj{\"o}lin, \emph{Regularity of solutions to the {S}chr\"odinger equation},
  Duke Math. J. \textbf{55} (1987), no.~3, 699--715. 

\bibitem{TaVa2000}
T.~Tao and A.~Vargas, \emph{A bilinear approach to cone multipliers. {II}.
  {A}pplications}, Geom. Funct. Anal. \textbf{10} (2000), no.~1, 216--258.

\bibitem{Vega88}
Luis Vega, \emph{Schr\"odinger equations: pointwise convergence to the initial
  data}, Proc. Amer. Math. Soc. \textbf{102} (1988), no.~4, 874--878.

\end{thebibliography}
\end{document}